\documentclass[12pt]{article}

\usepackage{amsmath,amssymb,amsbsy,amsfonts,amsthm,latexsym,
            amsopn,amstext,amsxtra,amscd,stmaryrd}

\usepackage[mathscr]{eucal}

\usepackage[english]{babel}
\usepackage{url}
\usepackage[colorlinks,linkcolor=blue,anchorcolor=blue,citecolor=blue]{hyperref}

\usepackage{color}
\usepackage[usenames,dvipsnames,svgnames,table]{xcolor}

\begin{document}

\newtheorem{theorem}{Theorem} 
\newtheorem{lemma}[theorem]{Lemma}
\newtheorem{cor}[theorem]{Corollary}
\newtheorem{proposition}{Proposition}

\theoremstyle{definition}
\newtheorem*{definition}{Definition}
\newtheorem*{remark}{Remark}
\newtheorem*{example}{Example}


\def\cA{\mathcal A}
\def\cB{\mathcal B}
\def\cC{\mathcal C}
\def\cD{\mathcal D}
\def\cE{\mathcal E}
\def\cF{\mathcal F}
\def\cG{\mathcal G}
\def\cH{\mathcal H}
\def\cI{\mathcal I}
\def\cJ{\mathcal J}
\def\cK{\mathcal K}
\def\cL{\mathcal L}
\def\cO{\mathcal O}
\def\cP{\mathcal P}
\def\cQ{\mathcal Q}
\def\cR{\mathcal R}
\def\cS{\mathcal S}
\def\cU{\mathcal U}
\def\cT{\mathcal T}
\def\cV{\mathcal V}
\def\cW{\mathcal W}
\def\cX{\mathcal X}
\def\cY{\mathcal Y}
\def\cZ{\mathcal Z}


\def\sA{\mathscr A}
\def\sB{\mathscr B}
\def\sC{\mathscr C}
\def\sD{\Delta}
\def\sE{\mathscr E}
\def\sF{\mathscr F}
\def\sG{\mathscr G}
\def\sH{\mathscr H}
\def\sI{\mathscr I}
\def\sJ{\mathscr J}
\def\sK{\mathscr K}
\def\sL{\mathscr L}
\def\sM{\mathscr M}
\def\sN{\mathscr N}
\def\sO{\mathscr O}
\def\sP{\mathscr P}
\def\sQ{\mathscr Q}
\def\sR{\mathscr R}
\def\sS{\mathscr S}
\def\sU{\mathscr U}
\def\sT{\mathscr T}
\def\sV{\mathscr V}
\def\sW{\mathscr W}
\def\sX{\mathscr X}
\def\sY{\mathscr Y}
\def\sZ{\mathscr Z}


\def\fA{\mathfrak A}
\def\fB{\mathfrak B}
\def\fC{\mathfrak C}
\def\fD{\mathfrak D}
\def\fE{\mathfrak E}
\def\fF{\mathfrak F}
\def\fG{\mathfrak G}
\def\fH{\mathfrak H}
\def\fI{\mathfrak I}
\def\fJ{\mathfrak J}
\def\fK{\mathfrak K}
\def\fL{\mathfrak L}
\def\fM{\mathfrak M}
\def\fN{\mathfrak N}
\def\fO{\mathfrak O}
\def\fP{\mathfrak P}
\def\fQ{\mathfrak Q}
\def\fR{\mathfrak R}
\def\fS{\mathfrak S}
\def\fU{\mathfrak U}
\def\fT{\mathfrak T}
\def\fV{\mathfrak V}
\def\fW{\mathfrak W}
\def\fX{\mathfrak X}
\def\fY{\mathfrak Y}
\def\fZ{\mathfrak Z}


\def\C{{\mathbb C}}
\def\F{{\mathbb F}}
\def\K{{\mathbb K}}
\def\L{{\mathbb L}}
\def\N{{\mathbb N}}
\def\Q{{\mathbb Q}}
\def\R{{\mathbb R}}
\def\Z{{\mathbb Z}}


\def\eps{\varepsilon}
\def\mand{\qquad\mbox{and}\qquad}
\def\\{\cr}
\def\({\left(}
\def\){\right)}
\def\[{\left[}
\def\]{\right]}
\def\<{\langle}
\def\>{\rangle}
\def\fl#1{\left\lfloor#1\right\rfloor}
\def\rf#1{\left\lceil#1\right\rceil}
\def\le{\leqslant}
\def\ge{\geqslant}
\def\ds{\displaystyle}

\def\xxx{\vskip5pt\hrule\vskip5pt}
\def\imhere{ \xxx\centerline{\sc I'm here}\xxx }

\def\imtoo{ \xxx\centerline{\sc I'm too!}\xxx }

\newcommand{\comm}[1]{\marginpar{
\vskip-\baselineskip \raggedright\footnotesize
\itshape\hrule\smallskip#1\par\smallskip\hrule}}

\def\cc#1{\textcolor{Red}{#1}}


\def\e{{\mathbf e}}
\def\en{\e_n}


\title{\bf Fractional parts of Dedekind sums}

\author{
{\sc William D.~Banks} \\
{Department of Mathematics} \\
{University of Missouri} \\
{Columbia, MO 65211 USA} \\
{\tt bankswd@missouri.edu} 
\and
{\sc Igor E.~Shparlinski} \\
{Department of Pure Mathematics} \\
{University of New South Wales} \\
{Sydney, NSW 2052, Australia} \\
{\tt igor.shparlinski@unsw.edu.au}}

\date{\today}
\pagenumbering{arabic}

\maketitle

\begin{abstract}
Using a recent improvement by Bettin and Chandee to 
a bound of Duke, Friedlander and Iwaniec~(1997)
on double exponential sums with Kloosterman fractions, we establish
a uniformity of distribution result for the fractional parts of 
Dedekind sums $s(m,n)$ with $m$ and $n$ running 
over rather general sets. Our result extends earlier work of
Myerson (1988) and Vardi (1987). Using different techniques, we 
also study the least denominator of the collection
of Dedekind sums $\bigl\{s(m,n):m\in(\Z/n\Z)^*\bigr\}$
on average for $n\in[1,N]$.
\end{abstract}

\bigskip

{\bf MSC numbers:} 11D45, 11D72, 11L40

\medskip

{\bf Keywords:} Dedekind sums, uniform distribution, exponential sums, Kloosterman fractions.

\maketitle

\newpage

\section{Introduction}
\label{sec:intro}

For any integers $n\ge m \ge 1$ the \emph{Dedekind sum} $s(m,n)$ is defined by
$$
s(m,n)=\sum_{k\bmod n}
\(\mkern-7mu\(\frac{km}{n}\)\mkern-7mu\)
\(\mkern-5mu\(\frac{m}{n}\)\mkern-5mu\),
$$
where $(\mkern-2mu(t)\mkern-2mu)$ denotes the distance from
the real number $t$ to the closest integer.
Vardi~\cite[Theorem~1.2]{Var} observes that the
bound
$$
\sum_{n\le x}\sum_{\substack{1\le m\le n\\\gcd(m,n)=1}}
\e(12k\,s(m,n))\ll x^{3/2+o(1)}\qquad (x \to \infty),
$$
where 
$$
\e(t)=\exp(2\pi it)\qquad(t\in\R),
$$
is an easy consequence of the Weil bound on Kloosterman sums.
This implies that for any fixed
integer $k\ne 0$ the collection of fractional parts
$\{12 k\,s(m,n)\}$ with $n\in\N$ and $m\in(\Z/n\Z)^*$
is uniformly distributed over the interval $[0,1)$;
see~\cite[Theorem~1.3]{Var}.  Vardi further shows (cf.~\cite[Theorem~1.6]{Var})
that the fractional parts $\{\rho\,s(m,n)\}$ are uniformly
distributed over $[0,1)$ for every fixed real number $\rho\ne 0$.

Myerson~\cite[Theorem~3]{Mye} extends the latter result by
showing that for any fixed $\rho\ne 0$ the set of pairs 
$\bigl(m/n, \{\rho\,s(m,n)\}\bigr)$ with $n\in\N$ and $m\in(\Z/n\Z)^*$
is uniformly distributed over the square $[0,1)\times[0,1)$.
This can be naturally interpreted as a statement
about the number of fractional parts $\{\rho\,s(m,n)\}$
with $1\le n\le N$, $1 \le m \le L_n$ and $\gcd(m,n)=1$ that fall into
any given connected interval in $[0,1)$, where the numbers $L_n$ are arbitrary integers
for which the sequence $(L_n/n)_{n\in\N}$
has a positive limit.

In the present paper we give another extension
of~\cite[Theorem~1.6]{Var}.  More precisely, suppose we are given
a real number $\rho\ne 0$, positive integers $M\le N$, and two sequences
of integers $\cK=(K_n)$ and $\cL = (L_n)$ for which
\begin{equation}
\label{eq:SeqKM}
M \le K_n < K_n + L_n\le 2M \qquad (n\sim N),
\end{equation}
where the notation $n\sim N$ is used here and elsewhere as an abbreviation
for $N<n\le 2N$. Furthermore, suppose that we are given two sets 
\begin{equation}
\label{eq:SetsMN}
\fM\subseteq(M,2M]\mand\fN\subseteq(N,2N].
\end{equation}

For a given choice of the data $\fD=(\rho,M,N,\cK,\cL,\fM,\fN)$ as above, we use
$\sA_\fD(\lambda)$ to denote the number
of pairs $(m,n) \in \fM \times \fN$ such that
\begin{equation}
\label{eq:Pair}
K_n < m \le K_n+L_n, \qquad \gcd(m,n)=1,
\end{equation}
and for which 
$$
\{\rho\,s(m,n)\}\in[0,\lambda].
$$

We also denote by $\sN_\fD = \sA_\fD(1)$ the number of 
pairs $(m,n)\in\fM \times\fN$ satisfying only~\eqref{eq:Pair}
(in particular, note that $\sN_\fD=\sA_\fD(1)$ regardless of the value of $\rho$). 
It is natural to expect that if the sets $\fM$ and $\fN$ 
are reasonably dense in the intervals $(M,2M]$ and  $(N,2N]$,
respectively, and do not have any local obstructions to the 
condition $\gcd(m,n)=1$ (such as containing only even numbers), then one 
might expect that 
\begin{equation}
\label{eq:N approx}
\sN_\fD=N^{o(1)} \sum_{n\in\fN}\#\([K_n,K_n+L_n]\cap\fM\).
\end{equation}
For instance, using a version of the prime number theorem for short intervals
(see~\cite[Section~10.5]{IwKow} and also~\cite{BaHaPi})
the bound~\eqref{eq:N approx} holds if $\fM$ is the
\emph{set of prime numbers} and most of the interval lengths $L_n$ are reasonably
long relative to $M$.  A similar result can also be obtained when $\fM$ is 
the \emph{set of $Q$-smooth numbers} (in different ranges of the smoothness level $Q$
depending on the sizes of  the interval lengths $L_n$);
necessary results can be found in the surveys~\cite{Granv,HilTen}. 

If $\sN_\fD$ is sufficiently large
(in particular if~\eqref{eq:N approx} holds) then
it is reasonable to expect that 
$\sA_\fD(\lambda)\approx\lambda\,\sN_\fD$ in many situations.
To make a quantitative statement, 
we denote by $\sD_\fD$ the largest 
deviation of $\sA_\fD(\lambda)$ from its expected value
as $\lambda$ varies over the interval $[0,1]$; that is, we set
$$
\sD_\fD=
\sup_{\lambda\in[0,1]}
\bigl|\sA_\fD(\lambda)-\lambda\,\sN_\fD\bigr|.
$$
Here, we demonstrate how recent results of
Bettin and  Chandee~\cite{BettChan}, which improve earlier estimates
of Duke, Friedlander and Iwaniec~\cite{DuFrIw}, can be used
to estimate $\sD_\fD$ for a wide variety of the
data $\fD=(\rho,M,N,\cK,\cL,\fM,\fN)$.
To illustrate the ideas, we focus on the simplest case in which $\rho=12$;
however, our approach works in much greater generality.  The following
result is proved below in \S\ref{sec:thm1}.

\begin{theorem}
\label{thm:Discr}
Let $\fD=(12,M,N,\cK,\cL,\fM,\fN)$ be a given tuple of data 
with positive integers $M\le N$, two sequences
of integers $\cK=(K_n)$ and $\cL=(L_n)$ satisfying~\eqref{eq:SeqKM},
and two sets $\fM$ and $\fN$  satisfying~\eqref{eq:SetsMN}.
Then the bound
$$
\sD_\fD\ll  |\fM\times\fN|^{1/2}  M^{3/10}N^{13/20+o(1)}   
+\sN_\fD M^{1/2}N^{-1/2} $$
holds as $M\to\infty$.
\end{theorem} 

Using the bound $\sN_\fD \le |\fM\times\fN|$ we simplify Theorem~\ref{thm:Discr}
as follows.

\begin{cor}
\label{cor:Discr1}
Under the conditions of Theorem~\ref{thm:Discr} the bound 
$$
\sD_\fD\ll |\fM\times\fN|^{1/2}  M^{3/10}N^{13/20+o(1)} 
+ |\fM\times\fN| M^{1/2}N^{-1/2} 
$$
holds as $M\to\infty$.
\end{cor}

Moreover, using $|\fM\times\fN| \le MN$ we can simplify Theorem~\ref{thm:Discr}
further. 

\begin{cor}
\label{cor:Discr2}
Under the conditions of Theorem~\ref{thm:Discr} the bound 
$$
\sD_\fD\ll M^{4/5}N^{23/20}
+   M^{3/2}N^{1/2} 
$$
holds as $M\to\infty$.
\end{cor} 

In the case that $\sN_\fD = (MN)^{1+o(1)}$ one sees that Corollary~\ref{cor:Discr2}
improves the trivial bound $\sD_\fD \le \sN_\fD$ provided that the inequalities
$N^{3/4+\eps} \le M \le N^{1-\eps}$ 
hold with some fixed $\eps>0$ as $M\to\infty$.

In this paper we also study the distribution of the least denominator $q(n)$
of the Dedekind sums to modulus $n$.  More precisely, expressing each
Dedekind sum $s(m,n)$ as a reduced fraction $a(m,n)/q(m,n)$, let
$$
q(n)=\min\bigl\{q(m,n)~:~m\in(\Z/n\Z)^*\bigr\}.
$$
In \S\ref{sec:thm2} we prove the following result.

\begin{theorem}
\label{thm:Denom}  
We have
$$
\sum_{n=1}^N q(n)= \left(C+o(1)\right)
\frac{N^2}{(\log N)^{1/2}}\qquad(N\to\infty),
$$
where 
$$
C=\frac{3\sqrt{2}}{8\pi}
\prod_{p \equiv 1 \bmod 4}(1-p^{-2})^{-1}
\prod_{p \equiv 3 \bmod 4}(1-p^{-2})^{-1/2}.
$$
\end{theorem} 

For other  recent results about the
distribution and other properties of Dedekind sums,
we refer the interested reader 
to~\cite{AXZ1,AXZ2,Girst1,Girst2,Girst3,Girst4,Girst5,JRW,Tsu}. 

\section{Preliminaries}

Throughout the paper, the implied constants in the symbols ``$O$'' and
``$\ll$'' are absolute.  We recall that
the expressions $A \ll B$ and $A=O(B)$ are each equivalent to the
statement that $|A|\le cB$ for some constant $c$.

We use $\|\cdot\|_1$ and $\|\cdot\|_2$ to denote the $L^1$ and $L^2$ norms,
respectively, for finite sequences of complex numbers.

Given coprime integers $m,n\ge 1$ we use
$m_n^*$ to denote the unique integer defined by the conditions
$$
m m_n^*\equiv 1 \bmod n \mand 1 \le m_n^*\le n.
$$

As mentioned in \S\ref{sec:intro}, our investigation of the distribution of the fractional 
parts $\{\rho\,s(m,n)\}$ focuses on the special case $\rho=12$;
accordingly, following Girstmair~\cite{Girst2} we denote
$$
S(m,n) = 12\,s(m,n).
$$
The next well known result is due to Hickerson~\cite{Hic}.

\begin{lemma}
\label{lem:FracPart}
For any coprime integers $m,n\ge 1$ we have
$$
S(m,n)-\frac{m+m_n^*}{n}\in\Z. 
$$
\end{lemma}

We also need the following 
estimate which is the special case 
$A=1$ of the more general bound of Bettin and 
Chandee~\cite{BettChan} on their sum
$\cC_1(M,N,A;\beta,\nu)$, which is obtained in the proof 
of their main result;
see also~\cite[Theorem~6]{DuFrIw}.

\begin{lemma}
\label{lem:DFIa Gen}
For any positive integer $b$ and any complex numbers $\beta_n$, the sum
$$
\sC(M,N;\beta,b)=\sum_{m\sim M}
\biggl|\sum_{\substack{n\sim N\\\gcd(n,m)=1}}\beta_n
\,\e\(b\,\frac{m^*_n}{n}\)\biggr|^2
$$
is bounded by
\begin{align*}
\sC(M,N;\beta,b)\le\|\beta\|_2^2&\(\frac{b}{MN} + 1\)^{1/2}\\
&(MN^{3/4}+N^{7/4} +
M^{6/5}N^{7/10} + M^{3/5}N^{13/10}) (MN)^{ o(1)}
\end{align*}
as $MN\to \infty$. 
\end{lemma}

In the case $M \le N$, which is relevant to our 
 situation, Lemma~\ref{lem:DFIa Gen} simplifies as follows.

\begin{cor}
\label{cor:DFIa}
For any positive integer $b$ and any complex numbers $\beta_n$, for $M \le N^{1+o(1)}$
the sum, in the notation of Lemma~\ref{lem:DFIa Gen} we have 
$$
\sC(M,N;\beta,b)\le\|\beta\|_2^2\(\frac{b}{MN} + 1\)^{1/2} (N^{7/4} + M^{3/5}N^{13/10}) 
N^{ o(1)}
$$
as $N\to \infty$. 
\end{cor}

Given a sequence $\cG=(\gamma_j)_{j=1}^J$ of real numbers 
in the interval $[0,1)$ we use
$\Delta_\cG$ to denote its \emph{discrepancy}; this quantity is defined by
$$
\Delta_\cG = \sup_{\lambda\in[0,1]}
\bigl|J_\cG(\lambda) - \lambda J \bigr|,
$$
where $J_\cG(\lambda)$ denotes the cardinality of $\cG\cap[0,\lambda]$
(in particular, $J_\cG(1) = J$). 

The celebrated \emph{Erd{\H o}s--Tur{\'a}n inequality} (see~\cite{DrTi,KuNi})
provides a means for deriving distributional properties of a sequence
from nontrivial bounds on corresponding exponential sums. 

\begin{lemma}
\label{lem:ErdTur} For any integer $H \ge 1$ the discrepancy $\Delta_\cG$ of
the real sequence $\cG=(\gamma_j)_{j=1}^J$ satisfies the bound
$$
\Delta_\cG \ll  \frac{J}{H}+\sum_{h=1}^H\frac{1}{h}\biggl|
\sum_{j=1}^J\e(h\gamma_j)\biggr|. 
$$
\end{lemma}

Let $v_p(\cdot)$ be the standard $p$-adic valuation; that is,
for any integer $n\ge 1$, $v_p(n)$ is the largest integer $e$
for which $p^e\mid n$.
To study the distribution of the least denominator $q(n)$
of the Dedekind sums to modulus~$n$, 
we make use of the following explicit formula of Girstmair~\cite[Corollary~1]{Girst1}.

\begin{lemma}
\label{lem:qn}
For  any  positive integer $n$ we have
$$
q(n) = \left\{\begin{array}{ll}
q_0(n) &\quad\text{if $n$ is odd,}\\
2^{v_2(n)-1} q_0(n)&\quad\text{if $n$ is even,}
\end{array}\right.
$$
where 
$$
q_0(n) = \prod_{\substack{p\,\mid\,n\\p \equiv 3 \bmod 4}} p^{v_p(n)}.
$$
\end{lemma}

To prove Theorem~\ref{thm:Denom} we apply Lemma~\ref{lem:qn}
in conjunction with the following classical theorem of Wirsing~\cite{Wirs}.

\begin{lemma}
\label{lem:Wirs} Assume that a real-valued multiplicative function
$f(n)$ satisfies
the following conditions:
\begin{itemize}
\item[$(i)$] $f(n) \ge 0$ for all $n\in\N$;
\item[$(ii)$] for some constants $a,b>0$ with $b<2$ the inequality 
$f(p^\alpha)\le ab^\alpha$ holds for all primes $p$ and integers $\alpha\ge 2$;
\item[$(iii)$] there exists a constant $\nu > 0$ such that
$$
\sum_{p \le N} f(p)  = \left(\nu + o(1) \right) \frac{N}{\log N}
\qquad(N\to\infty).
$$
\end{itemize}
Then
$$
\sum_{n \le N} f(n) = \biggl(\frac{1} {e^{\gamma\nu}\Gamma(\nu)} +o(1)\biggr)
\frac{N}{\log N}\prod_{p\le N}\,  \sum_{\alpha=0}^\infty \frac{f(p^\alpha)}{p^\alpha}
\qquad(N\to\infty),
$$
where $\gamma$ is the Euler-Mascheroni constant, and
$\Gamma(\cdot)$ is the gamma function of Euler.
\end{lemma}

\section{Double sums with Kloosterman fractions}

In what follows, we use the notation
$$
\e_k(t)=\exp(2\pi it/k)\qquad(k\in\N,~t\in\R).
$$

In the next lemma, we establish a variant of~\cite[Theorem~2]{DuFrIw}; 
however, we use Corollary~\ref{cor:DFIa} to get a quantitatively 
stronger result. 

\begin{lemma}
\label{lem:DFIb}
Given arbitrary integers $a$ and $b\ne 0$,  
positive integers $M\le N$, two sequences
of integers $\cK=(K_n)$ and $\cL=(L_n)$ satisfying~\eqref{eq:SeqKM},
and two sets $\fM$ and $\fN$  satisfying~\eqref{eq:SetsMN}, the sum 
$$
\sS =\sum_{\substack{(m,n)\in\fM\times\fN\\K_n<m\le K_n+L_n\\\gcd(m,n)=1}}
\en(am+bm_n^*)
$$
satisfies the uniform bound
$$
\sS\ll|\fM|^{1/2}|\fN|^{1/2}\(\frac{b}{MN} + 1\)^{1/4}  (N^{7/8} + M^{3/10}N^{13/20}) 
N^{ o(1)}+|a|JM N^{-1}
$$
provided that $|a|M\le N$, where $J$ is the number of 
pairs $(m,n)\in\fM \times\fN$ that satisfy~\eqref{eq:Pair}.
\end{lemma}

\begin{proof} Since $|a|M\le N$ by hypothesis, the estimate $\en(am)=1+O(|a| M/N)$
holds uniformly for all terms in the sum under consideration, hence
\begin{equation}
\label{eq:S0}
\cS =  \cS_0 + O\( |a| JM N^{-1}\),
\end{equation}
where 
$$
\cS_0=\sum_{\substack{(m,n)\in\fM\times\fN\\K_n<m\le K_n+L_n\\\gcd(m,n)=1}}
\en(bm_n^*). 
$$

We can assume $b\ge 1$.
Since $L_n\le M$ for every $n\sim N$ by~\eqref{eq:SeqKM},
taking $M_0=(M-1)/2$ we have the following relation
for all $m\sim M$:
$$
M^{-1}\sum_{-M_0<c\le M_0}\sum_{K_n<k\le K_n+L_n}
\e_M(c(k-m))=\begin{cases}
1&\quad\hbox{if $K_n<m\le K_n+L_n$},\\
0&\quad\hbox{otherwise}.
\end{cases}
$$
It follows that
\begin{equation}
\label{eq:insect}
\sS_0=M^{-1}\sum_{-M_0<c\le M_0}\sS_0(c),
\end{equation}
where
$$
\sS_0(c)=\sum_{\substack{(m,n)\in\fM\times\fN \\\gcd(m,n)=1}}
\alpha_m^{(c)}\beta_n^{(c)}\,\en(bm_n^*)
$$
with
$$
\alpha_m^{(c)}=\e_M(-cm)\qquad(m\in\fM)
$$
and 
$$
\beta_n^{(c)}=\sum_{K_n<k\le K_n+L_n}\e_M(ck)\qquad(n\in\fN).
$$
We also put 
$$
\beta_n^{(c)}=0\qquad(n\sim N,~n\not\in\fN).
$$
Using Cauchy's inequality we see that
$$
\bigl|\sS(c)\bigr|^2
\le |\fM| \sum_{m\sim M}\biggl|\sum_{\substack{n\sim N\\\gcd(m,n)=1}}
\beta_n^{(c)}\,\en(bm_n^*)\biggr|^2 
=|\fM|\,\sC(M,N;\beta^{(c)},b),
$$
where $\sC(M,N;\beta^{(c)},b)$ is defined as in Lemma~\ref{lem:DFIa Gen}.
Applying the bound of Corollary~\ref{cor:DFIa} it follows that
$$
\sS_0(c)\ll |\fM|^{1/2} \|\beta^{(c)}\|_2 
\(\frac{b}{MN} + 1\)^{1/4} (N^{7/8} + M^{3/10}N^{13/20}) 
N^{ o(1)},
$$
hence by~\eqref{eq:insect} we have
$$
\sS_0\ll |\fM|^{1/2}\(\frac{b}{MN} + 1\)^{1/4} M^{-1} (N^{7/8} + M^{3/10}N^{13/20})
N^{ o(1)}
\sum_{-M_0<c\le M_0} \|\beta^{(c)}\|_2.
$$

We now recall the well-known bound
\begin{equation}
\label{eq:beta}
  \beta_n^{(c)} \ll \min\left \{L_n,\frac{M}{|c|}\right \}  , 
\end{equation}
which holds for any integer $c$, with $0 < |c| \le M_0$;
see~\cite[Bound~(8.6)]{IwKow}.  Since $L_n\le M$, from
\eqref{eq:beta} we immediately derive that 
$$
\|\beta^{(c)}\|_2 \ll\frac{M|\fN|^{1/2}}{|c|+1}\qquad(-M_0<c\le M_0),
$$
and thus
$$
\sS_0\ll |\fM|^{1/2}|\fN|^{1/2} \(\frac{b}{MN} + 1\)^{1/4}   (N^{7/8} + M^{3/10}N^{13/20}) 
N^{ o(1)}.
$$
Using this relation in~\eqref{eq:S0} we complete the proof.
\end{proof}

\section{Proof of Theorem~\ref{thm:Discr}}
\label{sec:thm1}

Let $\cJ$ be the set of pairs $(m,n)$ in $\fM\times\fN$
that satisfy~\eqref{eq:Pair}. Put $J=|\cJ|=\sN_\fD$.  Applying Lemma~\ref{lem:ErdTur} to the
sequence $\cG$ of fractional parts $\{S(m,n)\}$ with $(m,n)\in\cJ$,
we see that the bound
\begin{equation}
\label{eq:DKL+ET}
\begin{split}
\sD_\fD&\ll\frac{J}{H}+\sum_{h=1}^H\frac{1}{h}
\bigg|\sum_{(m,n)\in\cJ}\e(h\,S(m,n))\bigg|\\
&=\frac{J}{H}+\sum_{h=1}^H\frac{1}{h}
\bigg|\sum_{(m,n)\in\cJ}\e_n(hm+hm_n^*)\bigg|
\end{split}
\end{equation}
holds uniformly for any integer $H\in[1,N]$,
where we have used Lemma~\ref{lem:FracPart} in the second step.
Next, we apply Lemma~\ref{lem:DFIb} with $a=b=h$
to bound each inner sum on the right side of~\eqref{eq:DKL+ET},
and we sum over $h$; using 
 the notation $K=|\fM\times\fN|= |\fM||\fN|$ we see that 
\begin{align*}
\sD_\fD
\ll JH^{-1}+K^{1/2}(N^{7/8} + M^{3/10}N^{13/20})  N^{o(1)}\sum_{h=1}^H\frac{1}{h}&
 \(\frac{b}{MN} + 1\)^{1/4} \\
&\qquad  +HJM N^{-1} . 
\end{align*}
Clearly, we can assume that $H \le MN$ as otherwise the last term 
$HJM N^{-1}$ exceeds the trivial bound $\sD_\fD \le J$. Under this condition
the above bound simplifies as
$$
\sD_\fD\ll JH^{-1}  + 
 K^{1/2}  (N^{7/8} + M^{3/10}N^{13/20}) N^{o(1)} 
+HJM N^{-1} . 
$$

We now choose $H=\fl{(N/M)^{1/2}}$ and note that since $N \ge M$ we have 
$H \asymp  (N/M)^{1/2}$.
Hence, with this choice we derive that
\begin{equation}
\label{eq:D Prelim}
\sD_\fD\ll JM^{1/2} N^{-1/2}  + 
 K^{1/2}  (N^{7/8} + M^{3/10}N^{13/20})  N^{o(1)} . 
\end{equation}
Since $K \le MN$ we see that the bound~\eqref{eq:D Prelim}
can be  nontrivial only in the case that
$N^{7/8} \le K^{1/2} \le (MN)^{1/2}$ and thus only for $M \ge N^{3/4}$.
However, in this case we have $N^{7/8} \le M^{3/10}N^{13/20}$, and~\eqref{eq:D Prelim}
simplifies to
$$
\sD_\fD\ll JM^{1/2} N^{-1/2}  + 
 K^{1/2}  M^{3/10}N^{13/20+o(1)}.
$$
This concludes the proof. 

\section{Proof of Theorem~\ref{thm:Denom}} 
\label{sec:thm2}

Let $f(n)=q(n)/n$ for all $n\in\N$.  Note that $f(1)=1$, and
$$
f(n) = \delta_n
\prod_{\substack{p\,\mid\,n\\p\equiv 1\bmod 4}}p^{-v_p(n)}\qquad(n>1),
$$
where
$$
\delta_n=\begin{cases}
1&\quad\hbox{if $n$ is odd},\\
1/2&\quad\hbox{if $n$ is even}.
\end{cases}
$$
It is easy to see that $f(n)$ is a multiplicative function satisfying
the conditions of Lemma~\ref{lem:Wirs} with $\nu=1/2$. Since $\Gamma(1/2)=\sqrt{\pi}$
we have
\begin{align*}
\sum_{n \le N} \frac{q(n)}{n}=\sum_{n\le N}f(n)
& = \left(\frac{1} {e^{\gamma/2}\sqrt{\pi}} +o(1)\right)
\frac{N}{\log N}\prod_{p\le N}\,  \sum_{\alpha=0}^\infty \frac{f(p^\alpha)}{p^\alpha}\\
& = \left(\frac{1} {e^{\gamma/2}\sqrt{\pi}} +o(1)\right)
\frac{N}{\log N}\, Q_2(N)\,Q_{4,1}(N)\,Q_{4,3}(N),
\end{align*}
where
\begin{align*}
Q_2(N)& =   \frac{1}{4} + \sum_{\alpha=1}^\infty \frac{1}{2^{\alpha+1}} = \frac{3}{4},\\
 Q_{4,1}(N)   & =  \prod_{\substack{p\le N \\ p \equiv 1 \bmod 4}}\,  
\sum_{\alpha=0}^\infty \frac{1}{p^{2\alpha}}=
 \prod_{\substack{p\le N \\ p \equiv 1 \bmod 4}} \frac{1}{1-p^{-2}},\\
Q_{4,3}(N) & =  \prod_{\substack{p\le N \\ p \equiv 3 \bmod 4}}\,  
\sum_{\alpha=0}^\infty \frac{1}{p^{\alpha}}=
\prod_{\substack{p\le N \\ p \equiv 3 \bmod 4}} \frac{1}{1-p^{-1}}.
\end{align*}
Clearly, the product for $Q_{4,1}(N)$ converges to 
$$
\varpi_{4,1}    =  \prod_{p \equiv 1 \bmod 4}\frac{1}{1-p^{-2}}.
$$ 
Furthermore, by a result of Uchiyama~\cite{Uchi} we have
$$
Q_{4,3}(N) = (\varpi_{4,3} + o(1))(\log N)^{1/2},
$$
where
$$
\varpi_{4,3} = \biggl(\frac{2e^\gamma}{\pi}
\prod_{p \equiv 3\bmod 4}\frac{1}{1-p^{-2}}\biggr)^{1/2}. 
$$
Collecting the above results we deduce that
$$
\sum_{n \le N} \frac{q(n)}{n}=
\biggl(\frac{3\,\varpi_{4,1}\varpi_{4,3}}{4\,e^{\gamma/2}\sqrt{\pi}} +o(1)\biggr)
\frac{N}{(\log N)^{1/2}}.
$$
The result now follows by partial summation.

\section{Comments}
\label{sec:Comm}

We note that in the case that $m$ runs through  a sufficiently 
long interval of consecutive integers, that is, for pairs
$(m,n)$ with $m$ running through \emph{all} integers satisfying
\eqref{eq:Pair}, the standard application 
of the Weil bound (see, for example, \cite[Chapter~11]{IwKow})
leads to a stronger bound than that of Theorem~\ref{thm:Discr}. 
More specifically, this approach works when
$L_n \ge n^{1/2+\varepsilon}$ for most values of $n$ under consideration (with an arbitrary fixed $\varepsilon>0$).

On the other hand, using recent bounds of very short 
Kloosterman sums due to Bourgain and Garaev~\cite{BouGar1,BouGar2}
one can obtain similar results in some cases in which $\fM$
is a fairly short interval of consecutive integers. For instance,
this approach works when $K_n = 0$ and $L_n \ge n^{\varepsilon}$
for most values of $n$.  This approach,
however, yields only logarithmic savings
over the trivial bound. 

Of course, neither of the approaches just mentioned  can be
applied in the general setting considered in this paper,
in which the sets $\fM$ and $\fN$ are essentially arbitrary. 

\section{Acknowledgments}
 
This work was initiated during a very enjoyable stay of the second author
at the Max Planck Institute for Mathematics (Bonn), and it progressed
during an equally enjoyable visit of the first author to the University of New
South Wales (Sydney); the authors thank both institutions for their
support and warm hospitality.
During the course of this research, the second author was supported in part
by ARC grant DP130100237.

\end{document}